\documentclass[amsthm]{article}

\usepackage{amsfonts}
\usepackage{amssymb}
\usepackage{amsthm}
\usepackage{amsmath}

\usepackage{tikz}
\usetikzlibrary{positioning, arrows, matrix, decorations.pathreplacing}
\usepackage{bbm}

\title{An elementary proof of the symplectic spectral theorem}

\author{Camilo Sanabria Malag\' on\footnote{Partially supported by Vicerrector\'ia de Investigaciones de la Universidad de los Andes grant PEP P13.160422.030 FAPA-Camilo Sanabria.}\\Department of Mathematics\\ Universidad de los Andes, Bogot\'a, Colombia}
\date{}

\begin{document}

\theoremstyle{plain} \newtheorem{theorem}{Theorem}
\theoremstyle{plain} \newtheorem{proposition}[theorem]{Proposition}
\theoremstyle{plain} \newtheorem{lemma}[theorem]{Lemma}
\theoremstyle{plain} \newtheorem{corollary}[theorem]{Corollary}
\theoremstyle{plain} \newtheorem{example}{Example}
\theoremstyle{plain} \newtheorem{remark}[example]{Remark}

\maketitle

\begin{abstract}
The classical spectral theorem completely describes self-adjoint operators on finite dimensional inner product vector spaces as linear combinations of orthogonal projections onto pairwise orthogonal subspaces. We prove a similar theorem for self-adjoint operators on finite dimensional symplectic vector spaces over perfect fields. We show that these operators decompose according to a polarization, i.e. as the product of an operator on a lagrangian subspace and its dual on a complementary lagrangian. Moreover, if all the eigenvalues of the operator are in the base field, then there exist a Darboux basis such that the matrix representation of the operator is two-by-two blocks block-diagonal, where the first block is in Jordan normal form and the second block is the transpose of the first one.
\end{abstract}

\section{Introduction}

Let $V$ be a finite dimensional vector space and let $f$ be a linear operator on $V$. Given a basis, we can associate a square matrix to this linear transformation. Furthermore, to describe $f$, we can choose a particular basis such that the transformation matrix has a specific normal form.

If we endow $V$ with additional structure, a natural question arises: what is the relationship of $f$ with this structure? A classical result of this type is the celebrated spectral theorem which characterizes self-adjoint operators in finite dimensional inner product spaces as diagonalizable with respect to an orthonormal basis.

Let $V$ be a symplectic vector space over $\mathbb{R}$ or over $\mathbb{C}$. A corollary of the spectral theorem characterizes self-skew-adjoint operators on $V$, known as hamiltonian operators. These operators are central to the theory, for they form the Lie algebra of the Lie group of operators preserving the symplectic structure. The result states that hamiltonian operators can be seen as the composition of a self-adjoint operator with respect to an inner product and an operator whose square is the opposite of the identity. 

Now, let $V$ be a symplectic vector space over a perfect field. Given a lagrangian subspace we can identify $V$ as the product of this subspace and its dual. This is known as a polarization of $V$. In this paper we will show that self-adjoint operators on $V$ with respect to the symplectic form decompose in a similar manner, i.e. they are the product of an operator and its dual (Corollary \ref{thetheorem}). This corollary is a coordinate free version of the main result of this paper.

\begin{theorem}
Let $\mathbb{K}$ be a field with $\textrm{char}(\mathbb{K})\ne 2$. Let $n\in\mathbb{Z}_{>0}$ and let $\Omega\in M_{2n\times 2n}(\mathbb{K})$ be the $2\times 2$-blocks matrix
$$ \Omega=\left[\begin{array}{cc} O_n & I_n\\ -I_n & O_n \end{array}\right], $$
where $I_n\in M_{n\times n}(\mathbb{K})$ is the identity matrix and $O_n\in M_{n\times n}(\mathbb{K})$ is the zero matrix. If $A\in M_{2n\times 2n}(\mathbb{K})$ is such that $$A^T=\Omega A\Omega^{-1},$$
then
\begin{itemize}
\item[i)] If all the eigenvalues of $A$ are in $\mathbb{K}$, then there exist a symplectic matrix $C$ such that
$$ C^{-1}AC=\left[\begin{array}{cc} B & O_n\\ O_n & B^T \end{array}\right], $$
for some $B\in M_{n\times n}(\mathbb{K})$ in Jordan normal form.
\item[ii)] If $\mathbb{K}$ is perfect, then there exist a symplectic matrix $C$ such that
$$ C^{-1}AC=\left[\begin{array}{cc} B & O_n\\ O_n & B^T \end{array}\right], $$
for some $B\in M_{n\times n}(\mathbb{K})$.
\end{itemize} 
\end{theorem}

A slightly different result has already been proven for an arbitrary field $\mathbb{K}$ with $\textrm{char}(\mathbb{K})\ne 2$ in \cite{S} using quivers. An alternative proof of that result was given in \cite{LR} for the case $\mathbb{K}=\mathbb{R}$ using pencils. Here we present a constructive proof using only elementary methods. We obtain that in the case when $\mathbb{K}$ is perfect and $\textrm{char}(\mathbb{K})\ne 2$ the block diagonal form can be obtained with respect to a Darboux basis. Furthermore, we consider the case when the base field $\mathbb{K}$ contains all the eigenvalues of the operator and we show that in this case we can choose the Darboux basis so that the first block is in Jordan normal form and the second block is the transpose of the first one.

\subsection{Symplectic adjoint}

Let $\mathbb{K}$ be a field with $\textrm{char}(\mathbb{K})\ne 2$ and let $V$ be a $\mathbb{K}$-vector space. A symplectic form over $V$ is an alternating non-degenerate bilinear form $\sigma:V\times V\rightarrow\mathbb{K}$. The pair $(V,\sigma)$ is called a symplectic vector space. 

Now, let us assume that $V$ has a finite dimension. Then its dimension is even, say $2n$, for some $n\in\mathbb{Z}_{>0}$ and $V$ admits a Darboux basis $\{u_1,\ldots,u_n,w_1,\ldots,w_n\}$ where $\sigma(u_i,u_j)=0$, $\sigma(w_i,w_j)=0$, and  $\sigma(w_i,u_j) =\delta_{ij}$  for $i,j\in\{1,\ldots,n\}$. Here $\delta_{ij}=1$ if $i=j$ and $\delta_{ij}=0$ if $i\ne j$. For any $v\in V$ we have
\[
v=\sum_{i=1}^n\sigma(w_i,v)u_i-\sigma(u_i,v)w_i.
\]

Let $f:V\rightarrow V$ be a $\mathbb{K}$-linear map. We say that a $\mathbb{K}$-linear map $g:V\rightarrow V$ is the adjoint of $f$ if $\sigma(v,f(w))=\sigma(g(v),w)$ for all $v,w\in V$.
It is defined by
$$g(v)=\sum_{i=1}^n\sigma(f(w_i),v)u_i-\sigma(f(u_i),v)w_i,$$
for $v\in V$.
If $g=f$, then $f$ is a self-adjoint operator.

\begin{lemma}\label{canexample}
Let $U$ be an $n$-dimensional $\mathbb{K}$-vector space and let $U^*$ be its dual. We consider the symplectic vector space $(U\times U^*,\omega)$ where for $v,w\in U$, $\lambda,\mu\in U^*$
$$\omega\Big((v,\lambda),(w,\mu)\Big)=\lambda(w)-\mu(v).$$
If $l:U\rightarrow U$ is a $\mathbb{K}$-linear transformation, then 
\begin{align*}
(l,l^*): U\times U^* & \longrightarrow  U\times U^*\\
 (v,\lambda) & \longmapsto  (l(v),l^*(\lambda))
\end{align*}
is a self-adjoint operator.
\end{lemma}

\begin{proof} For any $v,w\in U$, $\lambda,\mu\in U^*$ we have
\begin{align*}
\omega\Big((l(v),l^*(\lambda)),(w,\mu)\Big) &= l^*(\lambda)(w)-\mu(l(v))\\
 &= \lambda(l(w))-l^*(\mu)(v)\\
 &= \omega\Big((v,\lambda),(l(w),l^*(\mu))\Big).
\end{align*}
\end{proof}

\begin{lemma}\label{remarkcanexample}
Let $(V,\sigma)$ be a finite dimensional symplectic $\mathbb{K}$-vector space with $\dim(V)=2n$, $n\in\mathbb{Z}_{>0}$ and let $\{u_1,\ldots,u_n,w_1,\ldots,w_n\}$ be a Darboux basis of $V$. Let $U=\textrm{Span}(u_1,\ldots,u_n)$ and let $\{\lambda_1,\ldots,\lambda_n\}$ be the basis of $U^*$ dual to $\{u_1,\ldots,u_n\}$, i.e. $\lambda_i(u_j)=\delta_{ij}$. Then the $\mathbb{K}$-linear transformation
\[
\begin{array}{rcl}
\Phi: U\times U^* & \longrightarrow & V \\
        (u_i,0) & \longmapsto & u_i \\
        (0,\lambda_i) & \longmapsto & w_i,
\end{array}
\]
for $i=1,\ldots,n$, is an isomorphism between the symplectic $\mathbb{K}$-vector spaces $(U\times U,\omega)$ and $(V,\sigma)$, where $\omega$ is the symplectic form from Lemma \ref{canexample}.
\end{lemma}

\begin{proof} For any $v,w\in U$, $\lambda,\mu\in U^*$ we have
$$\sigma\big(\Phi(v,\lambda),\Phi(w,\mu)\big)=\omega\Big((v,\lambda),(w,\mu)\Big).$$ 
\end{proof}

In this paper we show that if $\mathbb{K}$ is a perfect field,  then every self-adjoint operator on a finite dimensional symplectic $\mathbb{K}$-vector space can be seen, up to an isomorphism of symplectic vector spaces, as the operator in Lemma \ref{canexample} for some $\mathbb{K}$-linear transformation $l$.

\section{Preliminary}

Let $(V,\sigma)$ be a symplectic vector space and let $\dim(V)=2n$, $n\in\mathbb{Z}_{>0}$. Given a subspace $W\subseteq V$, we define its $\sigma$-orthogonal complement by
\[
W^\sigma=\{v\in V|\ \sigma(v,w)=0\ \forall w\in W\}.
\]
Note that $\dim(W)+\dim(W^\sigma)=\dim(V)$.

We say that $W$ is a symplectic subspace of $V$ if $(W,\sigma|_W)$ is a symplectic space, where $\sigma|_W$ denotes the restriction of $\sigma$ to $W\times W$.

\begin{lemma}\label{lemma0}
$W\subseteq V$ is a symplectic subspace if and only if $W\cap W^\sigma=\{0\}$. In particular, $W\subseteq V$ is a symplectic subspace if and only if $W^\sigma$ is a symplectic subspace.
\end{lemma}

\begin{proof}
The form $\sigma|_W$ inherits from $\sigma$ the properties of being bilinear and alternating. It suffices to show that it is non-degenerate if and only if  $W\cap W^\sigma=\{0\}$. 

We first assume that $\sigma|_W$ is non-degenerate. If $v\in W\cap W^\sigma$, then for every $w\in W$ we have $\sigma(v,w)=0$, hence $v=0$. 

Conversely, we suppose that $W\cap W^\sigma=\{0\}$. Since $\dim(W)+\dim(W^\sigma)=\dim(V)$, we have $V=W\oplus W^\sigma$. Now, given $v\in W$, $v\ne 0$, there exists $w\in V$ such that $\sigma(v,w)\ne 0$. In particular, if $w=w_0+w_1$ with $w_0\in W$ and $w_1\in W^\sigma$, then $\sigma(v,w_0)=\sigma(v,w)\ne0$. 

The second claim follows from the first and the fact that $\left(W^\sigma\right)^\sigma=W$.
\end{proof}

\subsection{Symplectic invariant subspaces}

Let $\{u_1,\ldots,u_n,w_1,\ldots,w_n\}$ be a Darboux basis of $V$ and let $J\subseteq\{1,\ldots,n\}$. Then $V_J=\textrm{Span}(u_j,w_j)_{j\in J}$ is a symplectic subspace of $V$ and $\sigma|_{V_J}$ is non-degenerate. Moreover, if $J_1, \ldots, J_r\subseteq\{1,\ldots,n\}$ are such that $J_1\cup\ldots\cup J_r=\{1,\ldots,n\}$ and $J_i\cap J_j=\emptyset$ whenever $i\ne j$, then
\[
V=V_{J_1}\oplus\ldots\oplus V_{J_r}.
\]

In the following lemma we show that given a self-adjoint operator $f$ on $V$, we can decompose $V$ into $f$-invariant symplectic subspaces that correspond to the irreducible factors of the characteristic polynomial of $f$. 

\begin{lemma}\label{sympdecom}
Let $(V,\sigma)$ be a finite dimensional symplectic $\mathbb{K}$-vector space with $\dim(V)=2n$, $n\in\mathbb{Z}_{>0}$ and let $f:V\rightarrow V$ be a self-adjoint operator. Then
\begin{itemize}
\item[i)]  If $f$ is a projection, i.e. $f^2=f$, then $f(V)$ is a symplectic subspace.
\item[ii)] Let $P(t)=\sum_{i=0}^da_it^i\in\mathbb{K}[t]$ be a polynomial and let $P(f)=\sum_{i=0}^da_if^i$. Then $P(f)$ is self-adjoint.
\item[iii)] Let $P_f(t)=\det(f-t\mathbbm{1}_V)$ be the characteristic polynomial of $f$ and let $P_1(t),\ldots,P_r(t)\in\mathbb{K}[t]$ be the monic irreducible, pairwise coprime factors of $P_f(t)$ so that
\[
P_f(t)=\prod_{i=1}^r P_i(t)^{m_i},\quad m_1,\ldots,m_r\in\mathbb{Z}_{>0}.
\]
Then for $i=1,\ldots,r$, $V_i=\ker(P_i(f)^{m_i})$ is an $f$-invariant symplectic subspace, i.e. $f(V_i)\subseteq V_i$ and
\[
V=V_1\oplus\ldots\oplus V_r.
\]
\end{itemize}
\end{lemma}

\begin{proof}

\begin{enumerate}
\item[i)] We have $V=f(V)\oplus\ker(f)$. By Lemma \ref{lemma0} it suffices to show that $\ker(f)=f(V)^\sigma$, for this implies $f(V)\cap f(V)^\sigma=\{0\}$.

If $v\in\ker(f)$, then for every $w\in V$ we have
\[
\sigma(f(w),v)=\sigma(w,f(v))=0.
\]
Hence $v\in f(V)^\sigma$, and so $\ker(f)\subseteq f(V)^\sigma$. Since $\dim(\ker(f))=V-\dim(f(V))=\dim(f(V)^\sigma)$ we obtain $\ker(f)=f(V)^\sigma$. 
\item[ii)] For every $v,w\in V$ we have
\begin{align*}
\sigma\big(v,P(f)(w)\big) &=\sum_{i=0}^da_i\sigma\big(v,f^i(w)\big)\\
  &=\sum_{i=1}^da_i\sigma\big(f^i(v),w\big)\\
  &=\sigma\big(P(f)(v),w\big).
\end{align*}
\item[iii)] Note that for any two polynomials $Q_1(t), Q_2(t)\in\mathbb{K}[t]$ we have $Q_1(f)\circ Q_2(f)=Q_2(f)\circ Q_1(f)$.

Let $R_i(t)=\prod_{j\ne i}P_j(t)^{m_j}$ for $i=1,\ldots,r$. We have $\textrm{gcd}\big(R_1(t),\ldots,R_d(t)\big)=1$. Let $Q_1(t)\ldots,Q_r(t)\in\mathbb{K}[t]$ be such that $\sum_{j=1}^r Q_j(t)R_j(t)=1$. Let $p_i=Q_i(f)\circ R_i(f)$. Then $\sum_{j=1}^r p_j=\mathbbm{1}_V$.\\
For every $i,j\in\{1,\ldots,r\}$, $i\ne j$ we have $Q_i(t)R_i(t)Q_j(t)R_j(t)=Q_{ij}(t)P_f(t)$ for some $Q_{ij}(t)\in\mathbb{K}[t]$. In particular $p_i\circ p_j=Q_{ij}(f)\circ P_f(f)=0$. Hence
\[
p_i=p_i\circ\mathbbm{1}_V=p_i\circ \sum_{j=1}^r p_j=\sum_{i=1}^r p_i\circ p_j=p_i^2.
\]
Thus $p_i$ is a projection and
\[
V=p_1(V)\oplus\ldots\oplus p_r(V).
\]
We have $P_i^{m_i}(t)R_i(t)=P_f(t)$ and
\[
P_i(f)^{m_i}\circ p_i=P_i(f)\circ Q_i(f)\circ R_i(f)=Q_i(f)\circ P_f(f)=0,
\]
hence $p_i(V)\subseteq V_i$. On the other hand, if $v\in V_i$, then $R_j(f)(v)=0$ for $j\ne i$, therefore $p_j(v)=0$ and $v=p_i(v)$. We have $p_i(V)=V_i$.\\
Since for $i=1,\ldots,r$ the operator $p_i$ is a polynomial in $f$, by ii) we have that $p_i$ is a self-adjoint projection. Furthermore i) implies that $V_i$ is a symplectic subspace of $V$. Since $p_i$ and $f$ commute, we obtain that $p_i(V)$ is $f$-invariant.
\end{enumerate}
\end{proof}

\subsection{Nilpotent self-adjoint operators}

Let $U$ be a subspace of $V$. We say that $U$ is isotropic if $U\subseteq U^\sigma$. If furthermore $U=U^\sigma$, then $U$ is a lagrangian subspace. Let $\{u_1,\ldots,u_n,w_1,\ldots,w_n\}$ be a Darboux basis of $V$ and let $J\subseteq\{1,\ldots,n\}$. Then $U_J=\textrm{Span}(u_j)_{j\in J}$ and $W_J=\textrm{Span}(w_j)_{j\in J}$ are isotropic subspaces. 

A subspace $U\subseteq V$ is called $f$-cyclic if there exists $u\in U$ such that $U=\textrm{Span}\big(f^i(u)\big)_{i\in\mathbb{Z}_{\ge 0}}$. In particular, $f$-cyclic subspaces are $f$-invariant. 

Let $f$ be a nilpotent self-adjoint operator on $V$ with nilpotency degree $d$. In the following lemma we show that for any $f$-cyclic subspace $U\subseteq V$ with $\dim(U)=d$ there exists an $f$-cyclic subspace $W\subseteq V$ with $\dim(W)=d$ such that $U\oplus W$ is a symplectic subspace of $V$.

\begin{lemma}\label{nilpotent}
Let $(V,\sigma)$ be a finite dimensional symplectic $\mathbb{K}$-vector space with $\dim(V)=2n$, $n\in\mathbb{Z}_{>0}$ and let $f:V\rightarrow V$ be a nilpotent self-adjoint operator with nilpotency degree $d$. Then
\begin{itemize}
\item[i)]  If $U\subseteq V$ is an $f$-cyclic subspace, then $U$ is isotropic.
\item[ii)] For any $f$-cyclic subspace $U\subseteq V$ with $\dim(U)=d$, there exists an $f$-cyclic subspace $W\subseteq V$ with $\dim(W)=d$ such that $U\cap W=\{0\}$ and $U\oplus W$ is a symplectic subspace of $V$. Furthermore there exists a Darboux basis $\{u_1,\ldots,u_d,w_1,\ldots,w_d\}$ of $U\oplus W$ where $\{u_1,\ldots,u_d\}$ is a basis of $U$ and $\{w_1,\ldots,w_d\}$ is a basis of $W$ such that
\[
u_i=f^{d-i}(u_d) \textrm{ and } w_i=f^{i-1}(w_1),\\
\]
for $i=1,\ldots,d$.
\end{itemize}
\end{lemma}

\begin{proof} 
\begin{enumerate}
\item[i)] Let $u\in V$ be such that $U=\textrm{Span}\big(u,f(u),\ldots,f^{m-1}(u)\big)$ and $f^m(u)=0$. Then 
\[
\sigma\big(f^i(u),f^j(u)\big)=\sigma\big(f^{i+j}(u),u\big)=\sigma\big(f^j(u),f^i(u)\big)=-\sigma\big(f^i(u),f^j(u)\big),
\]
for $i,j\in\{0,1,\ldots,m-1\}$. Since $\textrm{char}(\mathbb{K})\ne 2$, we have $\sigma\big(f^i(u),f^j(u)\big)=0$. Hence $\sigma(u_1,u_2)=0$ for $u_1,u_2\in U$ and thus $U$ is isotropic. 

\item[ii) ]Let $u\in V$ be such that $U=\textrm{Span}\big(u,f(u),\ldots,f^{d-1}(u)\big)$, $f^{d-1}(u)\ne 0$ and $f^d(u)=0$. Let us denote $u_i=f^{d-i}(u)$, $i=1,\ldots,d$. We have $f(u_1)=f^d(u)=0$ and $f^i(u_j)=u_{j-i}$ for $1\le i<j\le d-1$. 

Since $u_1\ne 0$ there exists $v\in V$ such that $\sigma(v,u_1)=1$. We have
\begin{align*}
\sigma\big(f^{i-1}(v),u_i\big)&=\sigma\big(v,f^{i-1}(u_i)\big)=\sigma(v,u_1)=1,
\end{align*}
for $i=1,\ldots,d$ and
\begin{align*}
\sigma\big(f^{j-1}(v),u_i\big)&=\sigma\big(v,f^{j-1}(u_i)\big)=\sigma\big(v,f^{(j-i)}(u_1)\big)=\sigma(v,0)=0,
\end{align*}
for $1\le i< j\le d$.

We define $v^{(1)},\ldots, v^{(d)}$ recursively
\[\left\{\begin{array}{rcl}
v^{(1)} & = & v\\
v^{(k+1)} & = & v^{(k)}-\sigma(v^{(k)},u_k)f^{k}(v^{(k)}), \quad k=1,\ldots,d-1.
\end{array}\right.
\]
We denote $w_1=v^{(d)}$ and $w_i=f^{i-1}(w_1)$ for $i=1,\ldots,d-1$. We will prove that $\{u_1,\ldots,u_d,w_1,\ldots,w_d\}$ is the desired Darboux basis.

Inductively for $k=1,\ldots,d-1$ we obtain
\begin{align*}
\sigma\big(v^{(1)},u_1\big) &= 1
\end{align*}
and
\begin{align*}
\sigma\big(v^{(k+1)},u_1\big) &=\sigma(v^{(k)},u_1)-\sigma(v^{(k)},u_k)\sigma\big(f^{k}(v^{(k)}),u_1\big)\\
 & = \sigma(v^{(k)},u_1)-\sigma(v^{(k)},u_k)\sigma\big(v^{(k)},f^{k}(u_1)\big)\\
 & =\sigma(v^{(k)},u_1)-\sigma(v^{(k)},u_k)\sigma\big(v^{(k)},0\big)\\
 & =1.
\end{align*}
Then, inductively for $i=1,\ldots,d-1$ and $k=i+1,\ldots,d-1$ we obtain
\begin{align*}
\sigma\big(v^{(i+1)},u_{i+1}\big) & =\sigma(v^{(i)},u_{i+1})-\sigma(v^{(i)},u_{i+1})\sigma\big(f^{i}(v^{(i)}),u_{i+1}\big) \\
 & =\sigma(v^{(i)},u_{i+1})-\sigma(v^{(i)},u_{i+1})\sigma\big(v^{(i)},f^{i}(u_{i+1})\big)\\
 & =\sigma(v^{(i)},u_{i+1})-\sigma(v^{(i)},u_{i+1})\sigma\big(v^{(i)},u_{1}\big)\\
 & =0
 \end{align*}
 and
\begin{align*}
\sigma\big(v^{(k+1)},u_{i+1}\big) & =\sigma(v^{(k)},u_{i+1})-\sigma(v^{(k)},u_{i+1})\sigma\big(f^k(v^{(k)}),u_{i+1}\big) \\
 & =0-0\ \sigma\big(v^{(k)},f^{k}(u_{i+1})\big)\\
 & =0.
\end{align*}
Taking $k=d-1$ we have
\begin{align*}
\sigma\big(v^{(d)},u_1\big) & =1
\end{align*}
and
\begin{align*}
\sigma\big(v^{(d)},u_{i+1}\big) & = 0, \quad i=1,\ldots,d-1.
\end{align*}
Since $\sigma\big(v^{(d)},u_1\big)=1$ we obtain
\begin{align}
\sigma\big(f^{i-1}(v^{(d)}),u_{i}\big) & = 1,\quad i=1,\ldots,d\label{eq3}
\end{align}
and
\begin{align}
\sigma\big(f^{j-1}(v^{(d)}),u_{i}\big) & = 0,\quad 1\le i< j\le d.\label{eq4}
\end{align}
Furthermore, for $1\le j< i\le d$ we have
\begin{align}
\sigma\big(f^{j-1}(v^{(d)}),u_{i}\big) &= \sigma\big(v^{(d)},f^{j-1}(u_{i})\big) =  \sigma\big(v^{(d)},u_{(i-j)+1}\big) = 0.\label{eq5}
\end{align}

We define $U=\textrm{Span}\big(u_1,\ldots,u_d\big)$ and $W=\textrm{Span}\big(w_1,\ldots,w_d\big)$. Then i) and equations (\ref{eq3}), (\ref{eq4}) and (\ref{eq5}) imply that $\{u_1,\ldots,u_d,w_1,\ldots,w_d\}$ is a Darboux basis of $U\oplus W$. 
\end{enumerate}
\end{proof}

\section{Dual Jordan Blocks}

\subsection{The nilpotent case}

In the following proposition we prove that if $f$ is a nilpotent self-adjoint operator on $V$, then there exists a Darboux basis such that the matrix representation of $f$ is a two-by-two blocks block-diagonal, where the first block is in Jordan normal form and the second block is the transpose of the first one.

\begin{proposition}\label{nilpotentcase}
Let $\mathbb{K}$ be a field with $\textrm{char}(\mathbb{K})\ne 2$. Let $(V,\sigma)$ be a finite dimensional symplectic $\mathbb{K}$-vector space with $\dim(V)=2n$, $n\in\mathbb{Z}_{>0}$ and let $f:V\rightarrow V$ be a nilpotent self-adjoint operator. Then there exists a Darboux basis $\{u_1,\ldots,u_n,w_1,\ldots,w_n\}$ of $V$ such that the matrix representation of $f$ with respect to this Darboux basis is of the form 
$$ A=\left[\begin{array}{cc} B & O_n\\ O_n & B^T \end{array}\right], $$
for some $B\in M_{n\times n}(\mathbb{K})$ in Jordan normal form.
\end{proposition}

\begin{proof}
We will use induction on $n$. Let $d$ be the degree of nilpontency of $f$. Lemma \ref{nilpotent} implies that $2d\le 2n$.

If $n=1$, then $d=1$ and therefore $f=0$. The proposition follows for any Darboux basis and $B=O_n$. 

Let us assume $n>1$. Let $U\subseteq V$ be an $f$-cyclic subspace  with $\dim(U)=d$. Lemma \ref{nilpotentcase} implies that there exists an $f$-cyclic subspace $W\subseteq V$ with $\dim(W)=d$ such that $U\oplus W$ is a symplectic subspace of $V$ and that there exists a Darboux basis $\{u_1,\ldots,u_d,w_1,\ldots,w_d\}$ of $U\oplus W$ where $\{u_1,\ldots,u_d\}$ is a basis of $U$ and $\{w_1,\ldots,w_d\}$ is a basis of $W$ such that
\[
u_i=f^{d-i}(u_d) \textrm{ and } w_i=f^{i-1}(w_1),
\]
for $i=1,\ldots,d$.

We denote $V_0=U\oplus W\subseteq V$. We will prove that $V_0^\sigma$ is $f$-invariant. Let us assume that this is not the case. Then there exist $v\in V_0^\sigma$ such that $f(v)\not\in V_0^\sigma$. Furthermore, there exist $v_0\in V_0$ such that $\sigma(v_0,f(v))\ne 0$. Since $f$ is self-adjoint, we have $\sigma(f(v_0),v)\ne 0$. Hence $f(v_0)\ne 0$. Since $V_0$ is $f$-invariant, we have $f(v_0)\in V_0$ and $\sigma(f(v_0),v)=0$. We obtain a contradiction. 

We have $\dim(V_0^\sigma)=2(n-d)$ and by lemma \ref{lemma0}, $V_0^\sigma$ is a symplectic space. 

We will apply the induction hypothesis to the restriction $f|_{V_0^\sigma}$ to obtain a Darboux basis $\{u_{d+1},\ldots,u_n,w_{d+1},\dots,w_n\}$ of $V_0^\sigma$ such that the matrix representation of $f|_{V_0^\sigma}$ with respect to this Darboux basis is of the form 
$$ A_0=\left[\begin{array}{cc} B_0 & O_{n-d}\\ O_{n-d} & B_0^T \end{array}\right], $$
for some $B_0\in M_{n\times n}(\mathbb{K})$ in Jordan normal form.

Taking the Darboux basis $\{u_1,\ldots,u_n,w_1,\dots,w_n\}$, the proposition follows with $B$ being a block diagonal matrix where the first block is a Jordan block and the second one is $B_0$.
\end{proof}

\subsection{Splitting characteristic polynomial}

In the following proposition we prove that if $f$ is a self-adjoint operator on $V$ with all its eigenvalues in the base field, then there exists a Darboux basis such that the matrix representation of $f$ is a two-by-two blocks block-diagonal, where the first block is in Jordan normal form and the second block is the transpose of the first one.

\begin{proposition}\label{splittingcase}
Let $\mathbb{K}$ be a field with $\textrm{char}(\mathbb{K})\ne 2$. Let $(V,\sigma)$ be a finite dimensional symplectic $\mathbb{K}$-vector space with $\dim(V)=2n$, $n\in\mathbb{Z}_{>0}$ and let $f:V\rightarrow V$ be a self-adjoint operator on $V$ with all its eigenvalues in $\mathbb{K}$. Then there exists a Darboux basis $\{u_1,\ldots,u_n,w_1,\ldots,w_n\}$ of $V$ such that the matrix representation of $f$ with respect to this Darboux basis is of the form 
$$ A=\left[\begin{array}{cc} B & O_n\\ O_n & B^T \end{array}\right], $$
for some $B\in M_{n\times n}(\mathbb{K})$ in Jordan normal form.
\end{proposition}

\begin{proof}
Let $$P_f(t)=\prod_{i=1}^r(t-\lambda_i)^{m_i}$$ be the characteristic polynomial of $f$, where $\lambda_1,\ldots,\lambda_r\in\mathbb{K}$ with $\lambda_i\ne\lambda_j$ whenever $i\ne j$ and $m_1,\ldots,m_r\in\mathbb{Z}_{>0}$ with $m_1+\ldots+m_r=n$.

We will use induction on $r$. If $r=1$ then $$P_f(t)=(t-\lambda)^{n},$$
for some $\lambda\in\mathbb{K}$. The Caley-Hamilton theorem implies that $V=\ker\big( (f-\lambda\mathbbm{1}_V)^n\big)$, where $\mathbbm{1}_V$ is the identity on $V$. Thus $g=f-\lambda\mathbbm{1}_V$ is nilpotent and by Lemma \ref{sympdecom} ii) is self-adjoint. Proposition \ref{nilpotentcase} implies that there exists a Darboux basis $\{u_1,\ldots,u_n,w_1,\ldots,w_n\}$ of $V$ such that the matrix representation of $g$ with respect to this Darboux basis is of the form 
$$ A_0=\left[\begin{array}{cc} B_0 & O_n\\ O_n & B_0^T \end{array}\right], $$
for some $B_0\in M_{n\times n}(\mathbb{K})$ in Jordan normal form. The case $r=1$ follows by noting that if $A$ is the matrix representation of $f$ with respect to the Darboux basis $\{u_1,\ldots,u_n,w_1,\ldots,w_n\}$, then
$$ A=\left[\begin{array}{cc} B & O_n\\ O_n & B^T \end{array}\right], $$
where $B=B_0+\lambda I_n$ is in Jordan normal form.

Let $r>1$. Lemma \ref{sympdecom} iii) implies that $$V=V_1\oplus\ldots\oplus V_r,$$ where $V_i=\ker\big( (f-\lambda_i\mathbbm{1}_V)^m_i\big)$ is an $f$-invariant symplectic subspace for $i=1,\ldots,r$. Let us denote $V_0=V_2\oplus\ldots\oplus V_d$. The characteristic polynomial of the restriction $f_0=f|_{V_0}$ is $P_{f_0}(t)=\prod_{i=2}^r(t-\lambda_i)^{m_i}$. We apply the induction hypothesis to the restriction $f_0$ to obtain a Darboux basis $\{u_{m_1+1},\ldots,u_n,w_{m_1+1},\dots,w_n\}$ of $V_0$ such that the matrix representation of $f_0$ with respect to this Darboux basis is of the form 
$$ A_0=\left[\begin{array}{cc} B_0 & O_{n-{m_1}}\\ O_{n-{m_1}} & B_0^T \end{array}\right], $$
for some $B_0\in M_{(n-{m_1})\times (n-{m_1})}(\mathbb{K})$ in Jordan normal form.

We apply the case $r=1$ to the restriction $f_1=f|_{V_1}$ and we obtain a Darboux basis $\{u_1,\ldots,u_{m_1},w_1,\ldots,w_{m_1}\}$ of $V_1$ such that the matrix representation of $f_1$ with respect to this Darboux basis is of the form 
$$A_1=\left[\begin{array}{cc} B_1 & O_n\\ O_n & B_1^T \end{array}\right], $$
for some $B_1\in M_{m_1\times m_1}(\mathbb{K})$ in Jordan normal form. 

Taking the Darboux basis $\{u_1,\ldots,u_n,w_1,\dots,w_n\}$ the proposition follows with $B$ being the block diagonal matrix where the first block is $B_1$ and the second one is $B_0$.
\end{proof}

\subsection{Galois descent}

In this section we consider the case when the base field is perfect. We prove that if $f$ is a self-adjoint operator on $V$, then there exists a Darboux basis such that the matrix representation of $f$ is a two-by-two blocks block-diagonal, where the second block is the transpose of the first one.

\begin{proposition}\label{galoiscase}
Let $\mathbb{K}$ be a perfect field with $\textrm{char}(\mathbb{K})\ne 2$. Let $(V,\sigma)$ be a finite dimensional symplectic $\mathbb{K}$-vector space with $\dim(V)=2n$, $n\in\mathbb{Z}_{>0}$ and let $f:V\rightarrow V$ be a self-adjoint operator on $V$. Then there exists a Darboux basis $\{u_1,\ldots,u_n,w_1,\ldots,w_n\}$ of $V$ such that the matrix representation of $f$ with respect to this Darboux basis is of the form 
$$ A=\left[\begin{array}{cc} B & O_n\\ O_n & B^T \end{array}\right], $$
for some $B\in M_{n\times n}(\mathbb{K})$.
\end{proposition}

\begin{proof}
Let
\[
P_f(t)=\prod_{i=1}^r P_i(t)^{m_i},\quad m_1,\ldots,m_r\in\mathbb{Z}_>0
\]
be the characteristic polynomial of $f$, with $P_1(t),\ldots,P_r(t)\in\mathbb{K}[t]$ monic irreducible, pairwise coprime factors of $P_f(t)$. Lemma \ref{sympdecom} iii) implies that $$V=V_1\oplus\ldots\oplus V_r,$$ where $V_i=\ker\big(P_i(f)^{m_i}\big)$ is an $f$-invariant symplectic subspace for $i=1,\ldots,r$.

Following an argument similar to the proof of Proposition \ref{splittingcase} it suffices to proof the proposition in the case $r=1$. Thus we will assume that the characteristic polynomial of $f$, $P_f(t)\in \mathbb{K}[t]$, is of the form
\[
P_f(t)=P(t)^m,
\]
where $m\in\mathbb{Z}_{>0}$ and $P(t)\in\mathbb{K}[t]$ is an irreducible polynomial.

Let $\mathbb{K}_f$ be the splitting field of $P(t)$. Since $\mathbb{K}$ is perfect then $P(t)=\prod_{i=1}^d(t-\lambda_i)$ for some distinct $\lambda_1,\ldots,\lambda_d\in\mathbb{K}_f$. In particular
\[
\mathbb{K}_f=\mathbb{K}[\lambda_1,\ldots,\lambda_d].
\]
We denote $V_f=\mathbb{K}_f\otimes V$ and we define the $\mathbb{K}_f$-linear transformation $F:V_f\rightarrow V_f$ by
\[
F(c\otimes v)=c\otimes f(v),\quad c\in\mathbb{K}_f, v\in V.
\]
Its characteristic polynomial is $P_F(t)=(t-\lambda_1)^m\ldots(t-\lambda_d)^m$.

We define the $\mathbb{K}_f$-bilinear form $\sigma_f:V_f\times V_f\rightarrow\mathbb{K}_f$ by
\[
\sigma_f(c\otimes v,d\otimes w)=cd\ \sigma(v,w).
\]
Then $(V_f,\sigma_f)$ is a symplectic $\mathbb{K}_f$-vector space and $F$ is self-adjoint. Lemma \ref{sympdecom} iii) implies
\[
V_f=V_{f,1}\oplus\ldots\oplus V_{f,d},
\]
where $V_{f,i}=\ker\big((t-\lambda_i)^m\big)$ is an $F$-invariant symplectic subspace for $i=1,\ldots,d$.

Let $G$ be the Galois group of the field extension $\mathbb{K}_f/\mathbb{K}$. For each $\tau\in G$ we define the $\mathbb{K}$-linear isomorphism $\Psi_\tau:V_f\rightarrow V_f$ by
\[
\Psi_\tau(c\otimes v)=\tau(c)\otimes v,\quad c\in\mathbb{K}_f,v\in V.
\]
Note that $\omega_f\circ (\Psi_\tau\times\Psi_\tau)=\tau\circ\omega_f$ and $F\circ\Psi_\tau=\Psi_\tau\circ F$. Thus $\Psi_\tau(V_{f,i})=V_{f,\tau(i)}$, where $\lambda_{\tau(i)}=\tau(\lambda_i)$, $i=1,\ldots,d$.

We apply Proposition \ref{splittingcase} to the restriction $F_1=F|_{V_{f,1}}$ to obtain a Darboux basis $\{u^{(1)}_{f,1},\ldots,u^{(1)}_{f,m},w^{(1)}_{f,1},\ldots,w^{(1)}_{f,m}\}$ of $V_{f,1}$ such that the matrix representation of $F_1$ with respect to this Darboux basis is of the form 
$$ A_{f,1}=\left[\begin{array}{cc} B_{f,1} & O_n\\ O_n & B_{f,1}^T \end{array}\right], $$
for some $B_{f,1}\in M_{m\times m}(\mathbb{K}_f)$ in Jordan normal form.

For $i=2,\ldots,d$ let $\tau_i\in G$ be such that $\tau_i(\lambda_1)=\lambda_i$. We define $$u^{(i)}_{f,j}=\Psi_{\tau_i}(u^{(1)}_{f,j})\textrm{ and } w^{(i)}_{f,j}=\Psi_{\tau_i}(w^{(1)}_{f,j}),$$
for $j=1,\ldots,d$. Hence we obtain Darboux bases $\{u^{(i)}_{f,1},\ldots,u^{(i)}_{f,m},w^{(i)}_{f,1},\ldots,w^{(i)}_{f,m}\}$ of $V_{f,i}$ for $i=1\ldots,d$ such that the matrix representations of the restrictions $F_i=F|_{V_{f,i}}$ with respect to each of this Darboux bases are of the form 
$$ A_{f,i}=\left[\begin{array}{cc} B_{f,i} & O_n\\ O_n & B_{f,i}^T \end{array}\right], $$
for some $B_{f,i}\in M_{n\times n}(\mathbb{K}_f)$ in Jordan normal form.

We denote $U_{f,i}=\textrm{Span}\big(u^{(i)}_{f,1},\ldots,u^{(i)}_{f,m}\big)$ and $W_{f,i}=\textrm{Span}\big(w^{(i)}_{f,1},\ldots,w^{(i)}_{f,m}\big)$ for $i=1\ldots,d$. Furthermore we denote $U_{f}=U_{f,1}\oplus\ldots\oplus U_{f,d}$ and $W_{f}=W_{f,1}\oplus\ldots\oplus w_{f,d}$.
We have that $U_f$ and $W_f$ are $F$-invariant lagrangian subspaces such that for every $\tau\in G$, $\Psi_\tau (U_f)\subseteq U_f$ and $\Psi_\tau (W_f)\subseteq W_f$.  Therefore, if we apply Galois descent to the $G$-structure on $V_f$, we obtain two complementary $f$-invariant  lagrangian subspace $U$ and $W$ of $V$ and a Darboux basis $\{u_1,\ldots,u_n,w_1,\ldots,w_n\}$ of $V$ such that the matrix representation of $f$ with respect to this Darboux basis is of the form 
$$ A=\left[\begin{array}{cc} B & O_n\\ O_n & B^T \end{array}\right], $$
for some $B\in M_{n\times n}(\mathbb{K})$.
\end{proof}

\subsection{Symplectic spectral theorem}

Let $\{u_1,\ldots,u_n,w_1,\ldots,w_n\}$ be a Darboux basis of $V$. Let $\Omega\in M_{2n\times 2n}(\mathbb{K})$ be the $2\times 2$-blocks matrix
$$ \Omega=\left[\begin{array}{cc} O_n & I_n\\ -I_n & O_n \end{array}\right], $$
where $I_n\in M_{n\times n}(\mathbb{K})$ is the identity matrix and $O_n\in M_{n\times n}(\mathbb{K})$ is the zero matrix. Let $v,w\in V$ and let $\mathbf{x},\mathbf{y}\in\mathbb{K}^{2n}$ be the coordinate vectors of $v$ and $w$, respectively in the chosen Darboux basis. We have $$\sigma(v,w)=\mathbf{x}^T \Omega \mathbf{y}.$$ Let $A\in M_{2n\times 2n}(\mathbb{K})$ be the matrix representation of $f$ in the chosen Darboux basis. Then $f$ is self-adjoint if and only if $A^T\Omega=\Omega A$. Let $\{v_1,\ldots,v_{2n}\}$ be another basis of $V$ and let $C\in M_{2n\times 2n}(\mathbb{K})$ be the matrix of change of coordinates from this basis to the chosen Darboux basis. Then $\{v_1,\ldots,v_{2n}\}$ is a Darboux basis if and only if $C^T\Omega C=\Omega$. In this case, we say that $C$ is a symplectic matrix. We have the following theorem.

\begin{theorem}\label{matrixtheorem}
Let $\mathbb{K}$ be a field with $\textrm{char}(\mathbb{K})\ne 2$. Let $n\in\mathbb{Z}_{>0}$ and let $\Omega\in M_{2n\times 2n}(\mathbb{K})$ be the $2\times 2$-blocks matrix
$$ \Omega=\left[\begin{array}{cc} O_n & I_n\\ -I_n & O_n \end{array}\right], $$
where $I_n\in M_{n\times n}(\mathbb{K})$ is the identity matrix and $O_n\in M_{n\times n}(\mathbb{K})$ is the zero matrix. If $A\in M_{2n\times 2n}(\mathbb{K})$ is such that $$A^T=\Omega A\Omega^{-1},$$
then
\begin{itemize}
\item[i)] If all the eigenvalues of $A$ are in $\mathbb{K}$, then there exists a symplectic matrix $C$ such that
$$ C^{-1}AC=\left[\begin{array}{cc} B & O_n\\ O_n & B^T \end{array}\right], $$
for some $B\in M_{n\times n}(\mathbb{K})$ in Jordan normal form.
\item[ii)] If $\mathbb{K}$ is perfect, then there exists a symplectic matrix $C$ such that
$$ C^{-1}AC=\left[\begin{array}{cc} B & O_n\\ O_n & B^T \end{array}\right], $$
for some $B\in M_{n\times n}(\mathbb{K})$.
\end{itemize} 
\end{theorem}

\begin{proof}
Follows directly from Proposition \ref{splittingcase} and \ref{galoiscase}.
\end{proof}

\begin{corollary}\label{thetheorem}
Let $\mathbb{K}$ be a perfect field with $\textrm{char}(\mathbb{K})\ne 2$. Let $(V,\sigma)$ be a finite dimensional symplectic $\mathbb{K}$-vector space with $\dim(V)=2n$, $n\in\mathbb{Z}_{>0}$. Let $f:V\rightarrow V$ be a self-adjoint operator on $V$. Then there exists a Darboux basis $\{u_1,\ldots,u_n,w_1,\ldots,w_n\}$ of $V$ and a $\mathbb{K}$-linear transformation $l:U\rightarrow U$, where $U=\textrm{Span}(u_1,\ldots,u_n)$, such that (see Figure \ref{symspectral})
\[
\Phi\circ(l,l^*)=f\circ\Phi
\]
and $\Phi$ is the isomorphism defined in Lemma \ref{remarkcanexample}.
\end{corollary}

\begin{figure}[!hbp]
\centering
\begin{tikzpicture}[auto, node distance=3cm,>=latex']
    \node (UU1){$U\times U^*$};
    \node (V1) [below of=UU1] {$V$};
    \node (UU2) [right of=UU1] {$U\times U^*$};
    \node (V2) [right of=V1] {$V$};
    
    \path[->] (UU1) edge node {$\Phi$} (V1);
    \path[->] (V1) edge  node {$f$} (V2);
    \path[->] (UU2) edge  node {$\Phi$} (V2);
    \path[->] (UU1) edge  node {$(l,l^*)$} (UU2);
\end{tikzpicture}
\caption{Operator polarization}
\label{symspectral}
\end{figure}

\end{document}